\documentclass[a4paper,12pt]{article}
\usepackage{times}
\usepackage[english]{babel}
\usepackage{amssymb,amsmath}
\numberwithin{equation}{section}
\usepackage{amsthm}
\usepackage{array}
\usepackage{wasysym}
\usepackage[noadjust]{cite}
\usepackage{hyperref}
\usepackage[height=22.7cm , width = 16cm , top = 4cm , left = 3cm, a4paper]{geometry}
\usepackage{amssymb}
\usepackage{amsmath}
\usepackage{cite}
\usepackage{epsfig}
\usepackage{verbatim}
\usepackage{multicol}
\usepackage{graphicx, color, psfrag}
\usepackage{graphics, psfrag}
\usepackage{subfigure}
\usepackage[numbers,sort]{natbib}
 \newtheorem{theorem}{Theorem}[section]
\newtheorem{remark}[theorem]{Remark}

\newtheorem{corollary}{Corollary}[section]

\newtheorem{proposition}{Proposition}[section]
\theoremstyle{definition}

\newcommand{\e}{\end{document}}

\begin{document}

\thispagestyle{empty}

\author{
{{\bf B. S. El-Desouky\footnote{Corresponding author: b\_desouky@yahoo.com} \; and Abdelfattah Mustafa}
\newline{\it{{}  }}
 } { }\vspace{.2cm}\\
 \small \it Department of Mathematics, Faculty of Science, Mansoura University,Mansoura 35516, Egypt
}

\title{New Results and Matrix Representation for Daehee and Bernoulli Numbers and Polynomials}

\date{}

\maketitle
\small \pagestyle{myheadings}
        \markboth{{\scriptsize New results and matrix representation for Daehee and Bernoulli numbers and polynomials }}
        {{\scriptsize {El-Desouky and Mustafa}}}
%%%%%%%%%%%%%%%%%%%%%%%%%%%% Einleitung %%%%%%%%%%%%%%%%%

\hrule \vskip 8pt

\begin{abstract}
In this paper, we derive new matrix representation for Daehee numbers and polynomials, the $\lambda$-Daehee numbers and polynomials and the twisted Daehee numbers and polynomials. This helps us to obtain simple and short proofs of many previous results on Daehee numbers and polynomials. Moreover, we obtained some new results for Daehee and Bernoulli numbers and polynomials.
\end{abstract}

\noindent
{\bf Keywords:}
{\it Daehee numbers, Daehee polynomials, Bernoulli polynomials, Matrix representation.}

%%%%%%%%%%%%%%%%%%%%%%%%%%%%%%%%%%%%%%%%%%%%%%%%%%%%%%%%%%%%%%%%%%%%%%%%%%%%%%%%%%%%%%%%%%%%%%%%
%==========                                                                    ===============%%
%%                                1.Introduction                                              %%
%          ===================================================================                %%
%%%%%%%%%%%%%%%%%%%%%%%%%%%%%%%%%%%%%%%%%%%%%%%%%%%%%%%%%%%%%%%%%%%%%%%%%%%%%%%%%%%%%%%%%%%%%%%%

\section{Introduction}
As is known, the $n$-th Daehee polynomials are defined by the generating function, \cite{5}, \cite{7}, \cite{10}, \cite{12} and \cite{13},
\begin{equation} \label{1}
\left( \frac{\log⁡(1+t)}{t}\right) (1+t)^x= \sum_{n=0}^\infty D_n (x) \frac{ t^n}{n!}.
\end{equation}

\noindent
In the special case, $x=0,D_n=D_n (0)$ are called the Daehee numbers.\\
The Stirling numbers of the first kind are defined by
\begin{equation} \label{2}
(x)_n=\prod_{i=0}^n (x-i) =\sum_{l=0}^n s_1 (n,l) x^l,
\end{equation}

\noindent
and the Stirling numbers of the second kind  are given by the generating function to be,  \cite{2,3,4}.
\begin{equation} \label{3}
\left(e^t-1\right)^m = m! \sum_{l=m}^\infty s_2 (l,m) \frac{  t^l}{l!}.
\end{equation}

\noindent
For $\alpha \in \mathbb{Z}$, the Bernoulli polynomials of order $\alpha$ are defined by the generating function to be,  \cite{1,2} and \cite{10},
\begin{equation} \label{4}
\left( \frac{t}{e^t-1} \right)^{\alpha} e^{xt} = \sum_{n=0}^\infty B_n^{(\alpha) } (x) \frac{  t^n}{n!},
\end{equation}

\noindent
when $x=0, B_n^{(\alpha) }=B_n^{(\alpha) } (0)$ are called the Bernoulli numbers of order $\alpha$.

%%%%%%%%%%%%%%%%%%%%%%%%%%%%%%%%%%%%%%%%%%%%%%%%%%%%%%%%%%%%%%%%%%%%%%%%%%%%%%%%%%%%%%%%%%%%%%%%
%==========                                                                    ===============%%
%%                     2.Daehee Numbers and Polynomials                                       %%
%          ===================================================================                %%
%%%%%%%%%%%%%%%%%%%%%%%%%%%%%%%%%%%%%%%%%%%%%%%%%%%%%%%%%%%%%%%%%%%%%%%%%%%%%%%%%%%%%%%%%%%%%%%%

\section{Daehee Numbers and Polynomials}
In this section we derive an explicit formula and recurrence relation for Daehee numbers and polynomials of the first and second kinds and the relation between these numbers and N\"{o}rlund numbers are given.  D.S. Kim and T. Kim \cite{7}, obtained some results for Daehee numbers and polynomial of both kinds. We introduce the matrix representation and investigate a simple and short proofs of these results.\\

\noindent
From Eq. (\ref{1}), when $x=0$, the Daehee numbers of first kind are defined by
\begin{equation}\label{DN1}
\frac{\log⁡(1+t)}{t} = \sum_{n=0}^\infty D_n \frac{t^n}{n!},
\end{equation}

\noindent
by expanding the left hand side and equating the coefficients of $t^n$ on both sides gives
\begin{equation} \label{5}
D_n=(-1)^n   \frac{n!}{n+1}, \; \; n \geq 0.
\end{equation}

\noindent
It is easy to show that $D_n$ satisfy the recurrence relation
\begin{equation} \label{6}
(n+1) D_n+n^2 D_{n-1}=0, \; \;    n\geq 1.
\end{equation}

\noindent
For $x=-1$, the N\"{o}rlund numbers of the second kind have the explicit formula, see  Liu and Srivastava \cite[Remark 4]{9},
\begin{equation} \label{7}
b_n^{(-1) }=\frac{(-1)^n}{n+1},
\end{equation}

\noindent
thus we conclude that the relation between Daehee numbers and N\"{o}rlund numbers is given by
\begin{equation} \label{8}
D_n=n! b_n^{(-1)}.
\end{equation}

\begin{theorem} \label{th1}
For $n \geq 1, \; x \in \mathbb{Z} $  , we have
\begin{equation} \label{9}
D_n (1+x)=D_n (x)+nD_{n-1} (x).
\end{equation}
\end{theorem}
\begin{proof}
From Eq. (\ref{1}), replacing $x$ by $(1+x)$, we get
\begin{eqnarray}
\sum_{n=0}^\infty D_n (1+x)  \frac{ t^n}{n!} & = &
\left( \frac{\log⁡(1+t)}{t} \right) (1+t)^{1+x}=\sum_{n=0}^\infty D_n (x)   \frac{t^n}{n!}  (1+t)
\nonumber\\
& = & \sum_{n=0}^\infty D_n (x) \frac{t^n}{n!}+ \sum_{n=0}^\infty D_n (x) \frac{t^{n+1}}{n!}
\nonumber\\
& = & \sum_{n=0}^\infty D_n (x) \frac{t^n}{n!} + \sum_{n=0}^\infty (n+1) D_n (x) \frac{t^{n+1}}{(n+1)!}
\nonumber\\
& = & \sum_{n=0}^\infty D_n (x) \frac{ t^n}{n!} + \sum_{n=1}^\infty nD_{n-1} (x) \frac{t^n}{n!}.
\nonumber
\end{eqnarray}

\noindent
Equating the coefficients of $t^n$ on both sides gives (\ref{9}).
\end{proof}

\noindent
Similarly
\begin{equation} \label{10}
D_n (1-x)=D_n (-x)+nD_{n-1} (-x).
\end{equation}
\begin{corollary}
For $1 \leq k \leq n$, we have
\begin{equation} \label{11}
D_n (k)= \sum_{l=0}^k \left( \sum_{i=l}^k P(k,i) s_1 (i,l) D_{n-i} \right) n^l, \; \; \; k \leq n.
\end{equation}
\end{corollary}
\begin{proof}
From Eq. (\ref{9}), setting $x=0$, we have
\begin{equation*}
D_n (1)=D_n+n D_{n-1}.
\end{equation*}

\noindent
If $x=1$,
\begin{eqnarray}
D_n (2) & = & D_n (1)+ nD_{n-1} (1) =D_n + nD_{n-1} + n\left(D_{n-1}+(n-1) D_{n-2} \right)
\nonumber\\
& = & D_n+2n D_{n-1}+n(n-1) D_{n-2}.
\nonumber
\end{eqnarray}

\noindent
If $x=2$,
\begin{eqnarray}
D_n (3) & = & D_n (2) + nD_{n-1} (2)
\nonumber\\
& = & D_n + 2n D_{n-1} + n(n-1) D_{n-2}+n\left(D_{n-1}+2(n-1) D_{n-2}+(n-1)(n-2) D_{n-3} \right)
\nonumber\\
& = & D_n + 3n D_{n-1}+3n(n-1) D_{n-2}+n(n-1)(n-2) D_{n-3}.
\nonumber
\end{eqnarray}

\noindent
If $x=3$,
\begin{eqnarray}
D_n (4) & = & D_n (3)+nD_{n-1} (3)
\nonumber\\
& = & D_n + 3n D_{n-1}+3n(n-1) D_{n-2}+n(n-1)(n-2) D_{n-3}+ n\left(D_{n-1}+
\right.
\nonumber\\
&& \; \;
\left.
3(n-1) D_{n-2}+3(n-1)(n-2) D_{n-3}+(n-1)(n-2)(n-3) D_{n-4} \right)
\nonumber\\
& = & D_n + 4n D_{n-1}+6n(n-1) D_{n-2}+4n(n-1)(n-2) D_{n-3}+n(n-1)(n-2)(n-3) D_{n-4},
\nonumber\\
& = & \Big(^4_0\Big)D_n + \Big(^4_1\Big)(n)_1 D_{n-1}+\Big(^4_2\Big)(n)_2 D_{n-2}+\Big(^4_3\Big)(n)_3 D_{n-3}+\Big(^4_4\Big)(n)_4 D_{n-4},
\nonumber
\end{eqnarray}

\noindent
hence by iteration, for $x=k, \; 1 \leq k \leq n$, we have
\begin{eqnarray}
D_n (k)& = & \sum_{i=0}^k P(k,i) (n)_i D_{n-i} = \sum_{i=0}^k P(k,i) \sum_{l=0}^i s_1 (i,l) n^l D_{n-i}
\nonumber\\
& = & \sum_{l=0}^k \left( \sum_{i=l}^k P(k,i) s_1 (i,l) D_{n-i} \right) n^l , \; \; k \leq n.
\nonumber
\end{eqnarray}
This completes the proof.
\end{proof}
\begin{remark}
It is easy to show that $D_n (k)$ satisfies the symmetric relation
\begin{equation*}
D_n (k)=D_n (n-k), \; \; 1 \leq k \leq n.
\end{equation*}
\end{remark}

\noindent
We can write Eq. (\ref{11}), in the matrix form as follows.
\begin{equation} \label{12}
{\bf D} (k)= {\bf P}_{k+1} {\bf  D}_{n-k} {\bf S}_1 {\bf N},
\end{equation}

\noindent
where ${\bf D}(k)=\left(D_n (0) \; D_n (1) \; \cdots \; D_n (k) \right)^T$, is the Daehee polynomials for $x=k, \; k\in N$. ${\bf P}_{k+1}$ is the  $(k+1)\times (k+1)$ lower triangular Pascal matrix, ${\bf D}_{n-k}$, is the diagonal matrix of the Daehee number with elements $D_{n-i}$,  for $i=0,1, \cdots, k, \; 0\leq k\leq n$, ${\bf N}=\left( 1 \; n^1 \; n^2 \; \cdots \; n^k \right)^T$ is the $(k+1)\times 1$ matrix and ${\bf S}_1 $ is the $(k+1)\times(k+1)$, lower triangular matric for Stirling number of the first kind, see Comtet \cite{3} and El-Desouky et al. \cite{33}, i.e.
\begin{eqnarray*}
{\bf S}_1 & = &
\left(
\begin{array}{ccccc}
s_{0,0} & 0 & 0 & \cdots & 0 \\ s_{1,0} & s_{1,1} & 0 & \cdots & 0 \\ s_{2,0} & s_{2,1} & s_{2,2} & \cdots & 0 \\ \vdots & \vdots & \vdots &\ddots & \vdots \\ s_{n,0} & s_{n,1} & s_{n,2} & \cdots & s_{n,n} \\
\end{array}
\right)
\\
\left({\bf P}_{k+1} \right)_{ij} & =&
\left\{
\begin{array}{cl}
\Big(^i_j\Big), & \mbox{if} \; i\geq j, \; i,j=0,1, \cdots k,\\
0,              &  \mbox {otherwise.} \\
\end{array}
\right.
\end{eqnarray*}

\noindent
For example, if setting  $n=4, \; 0\leq k \leq n$ , in (\ref{12}), we have
\begin{eqnarray*}
&& \left(
\begin{array}{ccccc}
1 & 0 & 0 & 0 & 0 \\ 1 & 1 & 0 & 0 & 0 \\ 1 & 2 & 1 & 0 & 0 \\ 1 & 3 & 3 & 1 & 0 \\ 1 & 4 & 6 & 4 & 1 \\
\end{array}
\right)
\left(
\begin{array}{ccccc}
24/5 & 0 & 0 & 0 & 0\\ 0 & -3/2 & 0 & 0 & 0 \\ 0 & 0 & 2/3 & 0 & 0 \\ 0 & 0 & 0 & -1/2 & 0 \\ 0 & 0 & 0 & 0 & 1\\
\end{array}
\right)
\left(
\begin{array}{ccccc}
1 & 0 & 0 & 0 & 0\\ 0 & 1 & 0 & 0 & 0 \\ 0 & -1 & 1 & 0 & 0 \\ 0 & 2 & -3 & 1 &0 \\ 0 & -6 & 11 & -6 & 1\\
\end{array}
\right)
\left(
\begin{array}{c}
1 \\4\\16 \\64 \\256 \\
\end{array}
\right)
\\
&&
 \; \; \; \; \; \; \; \; \;  \; \; \; \; \; \; \; \; \; \; \; \; \; \; \; \; \; \;
 =\left(
\begin{array}{c}
24/5 \\ -6/5 \\ 4/5 \\ -6/5 \\ 24/5\\
\end{array}
\right).
\end{eqnarray*}

\noindent
D-S. Kim and T. Kim \cite[Eq. (2.10)]{7} proved the following relation
\begin{equation} \label{13}
D_n = \sum_{l=0}^n s_1 (n,l)  B_l.
\end{equation}

\noindent
We can write this relation in the matrix form as follows
\begin{equation} \label{14}
{\bf D}= \bf{S}_1 {\bf B},
\end{equation}

\noindent
where ${\bf D} =\left(D_0 \; D_1 \; \cdots D_n  \right)^T$   and  ${\bf B}=\left( B_0 \; B_1 \; \cdots \; B_n \right)^T,$ respectively, are the matrices for Daehee numbers of the first kind and Bernoulli numbers.

\noindent
For example, if  $0 \leq n \leq 4$ , in (\ref{14}), we have
\begin{equation*}
\left(
\begin{array}{ccccc}
1 & 0 & 0 & 0 & 0\\ 0 & 1 & 0 & 0 & 0 \\ 0 & -1 & 1 & 0 & 0 \\ 0 & 2 & -3 & 1 & 0 \\ 0 & -6 & 11 & -6 & 1
\end{array}
\right)
\left(
\begin{array}{c}
1 \\ -1/2 \\  1/6 \\ 0 \\ -1/30 \\
\end{array}
\right)=\left(
\begin{array}{c}
1 \\ -1/2 \\   2/3 \\ -3/2 \\ 24/5  \\
\end{array}
\right).
\end{equation*}

\noindent
D-S. Kim and T. Kim \cite[Corollary 3]{7} introduced the following result. \\
For $n \geq 0$, we have
\begin{equation} \label{15}
D_n (x)=\sum_{l=0}^n s_1 (n,l) B_l (x).
\end{equation}

\noindent
Eq. (\ref{15}) can be represented in the matrix form as follows
\begin{equation} \label{16}
{\bf D}(x)={\bf S}_1 {\bf B}(x),
\end{equation}

\noindent
where ${\bf D} (x)=\left( D_0 (x) \; D_1 (x) \; \cdots \; D_n (x)\right)^T$ and  ${\bf B}(x)=\left(B_0 (x) \; B_1 (x) \; \cdots \; B_n (x) \right)^T$ are the matrices for Daehee and Bernoulli polynomials.

\noindent
For example, if setting  $0 \leq n \leq 4$ in (\ref{16}), we have
\begin{equation*}
\left(
\begin{array}{ccccc}
1 & 0 & 0 & 0 & 0 \\ 0 & 1 & 0 & 0 & 0 \\ 0 & -1 & 1 & 0 & 0 \\ 0 & 2 & -3 & 1 & 0 \\ 0 & -6 & 11 & -6 & 1\\
\end{array}
\right)\left(
\begin{array}{c}
1 \\ x-1/2 \\ x^2-x+1/6 \\ x^3-3x^2/2+x/2 \\ x^4-2x^3+x^2-1/30
\end{array}
\right)=\left(
\begin{array}{c}
1 \\ x-1/2 \\ x^2-2x+2/3 \\ x^3-9x^2/2+11x/2-3/2 \\ x^4-8x^3+21x^2-20x+24/5\\
\end{array}
\right).
\end{equation*}

\noindent
D-S. Kim and T. Kim \cite[Theorem 4]{7} introduced the following results.\\
For $ m \geq 0$, we have
\begin{equation} \label{17}
B_m = \sum_{n=0}^m s_2 (m,n) D_n.
\end{equation}

\noindent
We can write Eq. (\ref{17}) in the matrix form as follows
\begin{equation} \label{18}
{\bf B} = {\bf S}_2 {\bf D},
\end{equation}

\noindent
where ${\bf S}_2$ is $(n+1)\times (n+1)$ lower triangular matrix for the Stirling numbers of the second kind.

\noindent
For example, if $0 \leq n \leq 4$, in (\ref{18}) we have
\begin{equation*}
\left(
\begin{array}{ccccc}
1 & 0 & 0 & 0 & 0 \\ 0 & 1 & 0 & 0 & 0 \\ 0 & 1 & 1 & 0 & 0 \\ 0 & 1 & 3 & 1 & 0 \\ 0 & 1 & 7 & 6 & 1 \\
\end{array}
\right)\left(
\begin{array}{c}
1 \\ -1/2 \\ 2/3 \\ -3/2 \\ 24/5 \\
\end{array}
\right) = \left(
\begin{array}{c}
1 \\ -1/2 \\ 1/6 \\ 0 \\ -1/30\\
\end{array}
\right).
\end{equation*}

\noindent
\begin{remark}
Using the matrix form (\ref{14}), we easily derive a short proof of Theorem 4 in Kim D-S. Kim and T. Kim \cite{7}. 
Multiplying both sides by the Striling number of second kind as follows
\begin{equation*}
{\bf S}_2 {\bf D} = {\bf S} _2 {\bf S }_1 {\bf B } = {\bf I B}= {\bf B},
\end{equation*}

\noindent
where ${\bf I}$ is the identity matrix of order $n+1$.
\end{remark}

\noindent
D-S. Kim and T. Kim \cite{7} defined the Daehee polynomials of the second kind by the generating function as follows.
\begin{equation} \label{19}
\frac{(1+t)  \log⁡(1+t)}{t} \frac{ 1}{(1+t)^x} = \sum_{n=0}^\infty \hat{D}_n (x) \frac{t^n}{n!}.
\end{equation}

\noindent
In the special case, $x=0, \; \hat{D}_n=\hat{D}_n (0)$ are called the Daehee numbers of the second kind
\begin{equation} \label{20}
\frac{ (1+t)  \log⁡(1+t)}{t} = \sum_{n=0}^\infty \hat{D}_n  \frac{t^n}{n!},
\end{equation}

\noindent
by expanding the left hand side and equating the coefficients of $t^n$ on both sides gives
\begin{equation}
\hat{D}_n=(-1)^{n-1} \frac{ n!}{n(n+1)} , \;    n\geq 1, \; \; \mbox{and} \;  \hat{D}_0 =1.
\end{equation}

\noindent
It is easy to show that $\hat{D}_n$ satisfy the recurrence relation
\begin{equation} \label{22}
(n+1)  \hat {D}_n+(n-1)n \hat{D}_(n-1)=0.
\end{equation}

\noindent
The relation between Daehee numbers of the first and the second kinds can be obtained as
\begin{equation} \label{23}
\hat{D}_n=D_n+n D_{n-1}, \; \; n\geq 1.
\end{equation}
and
\begin{equation} \label{24}
n \hat{D}_n + D_n=0, \; \; n \geq 1.
\end{equation}

\noindent
D-S. Kim and T. Kim \cite[Theorem 5]{7} introduced the following result.\\
For  $n \geq 0$, we have
\begin{equation} \label{25}
\hat{D}_n = \sum_{l=0}^n s_1 (n,l) (-1)^l B_l,
\end{equation}

\noindent
we can write this theorem in the matrix form as follows
\begin{equation} \label{26}
{\bf \hat{D}} = {\bf S}_1 {\bf I} _1 {\bf B},
\end{equation}

\noindent
where ${\bf \hat{D}}= \left( \hat{D}_0 \; \hat{D}_1 \; \cdots  \; \hat{D}_n \right)^T$, is the $(n+1)\times (n+1)$ matrix for Daehee numbers of the second kind, ${\bf I}_1$ is the $(n+1)\times(n+1)$ diagonal matrix its elements $({\bf I_1})_{ii}=(-1)^i, \; 0 \leq i \leq n$,  and ${ \bf B}$ is the $(n+1)\times 1$, matrix for Bernoulli numbers.

\noindent
For example, if setting $ 0 \leq n \leq 4$, in (\ref{26}), we have
\begin{equation*}
\left(
\begin{array}{ccccc}
1 & 0 & 0 & 0 & 0 \\ 0 & 1 & 0 & 0 & 0 \\ 0 & -1 & 1 & 0 & 0 \\ 0 & 2 & -3 & 1 & 0 \\ 0 & -6 & 11 & -6 & 1 \end{array}
\right)\left(
\begin{array}{ccccc}
1 & 0 & 0 & 0 & 0 \\ 0 & -1 & 0 & 0 & 0 \\ 0 & 0 & 1 & 0 & 0 \\ 0 & 0 & 0 & -1 & 0 \\ 0 & 0 & 0 & 0 & 1 \end{array}
\right)\left(
\begin{array}{c}
1 \\ -1/2 \\ 1/6 \\ 0 \\ -1/30 \\
\end{array}\right)=\left(
\begin{array}{c}
1 \\ 1/2 \\ -1/3 \\ 1/2 \\ -6/5
\end{array}
\right).
\end{equation*}

\noindent
D-S. Kim and T. Kim \cite[Theorem 6]{7} introduced the following result.\\
For  $ n \geq 0$, we have
\begin{equation} \label{27}
\hat{D}_n (x) = \sum_{l=0}^n s_1 (n,l)(-1)^l B_l (x).
\end{equation}

\noindent
We can write Eq. (\ref{27}) in the matrix form as follows
\begin{equation} \label{28}
{ \bf \hat{D}}̂(x)={\bf S}_1 {\bf I}_1 {\bf B}(x),
\end{equation}

\noindent
where  ${\bf \hat{D}}(x)=\left( \hat{D}_0 (x) \; \hat{D}_1 (x) \; \cdots \; \hat{D}_n (x) \right)^T$ is the $(n+1)\times 1$, matrix for Daehee polynomials of the second kind, ${\bf B} (x) $ is the $(n+1)\times 1$, matrix for Bernoulli polynomials.

\noindent
For example, if setting, $0 \leq n \leq 4$ in (\ref{28}), we have
\begin{eqnarray*}
&&
\left(
\begin{array}{ccccc}
1 & 0 & 0 & 0 & 0 \\ 0 & 1 & 0 & 0 & 0 \\ 0 & -1 & 1 & 0 & 0 \\ 0 & 2 & -3 & 1 & 0 \\ 0 & -6 & 11 & -6 & 1 \\
\end{array}
\right)\left(
\begin{array}{ccccc}
1 & 0 & 0 & 0 & 0 \\ 0 & -1 & 0 & 0 & 0 \\ 0 & 0 & 1 & 0 & 0 \\ 0 & 0 & 0 & -1 & 0 \\ 0 & 0 & 0 & 0 & 1 \\
\end{array}
\right) \left(
\begin{array}{c}
1 \\ x-1/2 \\ x^2-x+1/6 \\ x^3-3x^2/2+x/2 \\ x^4-2x^3+x^2-1/30 \\
\end{array}
\right)  =
\\
&&
 \; \; \; \; \; \; \; \; \;  \; \; \; \; \; \; \; \; \; \; \; \; \; \; \; \; \; \;
\left(
\begin{array}{c}
1 \\ -x+1/2 \\ x^2-1/3 \\ -x^3-3x^2/2+x/2+1/2 \\ x^4+4x^3+3x^2-2x-6/5
\end{array}
\right).
\end{eqnarray*}

\noindent
D-S. Kim and T. Kim \cite[Theorem 7]{7} introduced the following result.\\
For $m \geq 0$, we have
\begin{equation} \label{29}
B_m (1-x) = \sum_{n=0}^m s_2 (m,n) \hat{ D}_n (x),
\end{equation}

\noindent
where $B_n (1-x)=(-1)^n B_n (x)$.

\noindent
We can write Eq. (\ref{29}) in the matrix form as follows.
\begin{equation} \label{30}
{\bf B}(1-x)={\bf S}_2 {\bf \hat{D}}(x).
\end{equation}
\begin{remark} In fact we can prove Eq. (\ref{29}), D-S. Kim and T. Kim \cite[Theorem 7]{7} by multiplying Eq. (\ref{28}) by ${\bf S}_2$ as follows.
\begin{equation*}
{\bf S}_2 {\bf \hat{D}}(x)= {\bf S}_2 {\bf S}_1 {\bf I}_n {\bf B}(x)={\bf I}_1 {\bf B} (x)= {\bf B} (1-x).
\end{equation*}
\end{remark}

\noindent
For example, if setting $0 \leq n \leq 4$ in (\ref{30}), we have
\begin{equation*}
\left(
\begin{array}{ccccc}
1 & 0 & 0 & 0 & 0 \\ 0 & 1 & 0 & 0 & 0 \\ 0 & 1 & 1 & 0 & 0 \\ 0 & 1 & 3 & 1 & 0 \\ 0 & 1 & 7 & 6 & 1
\end{array}
\right)\left(
\begin{array}{c}
1 \\ -x+1/2 \\ x^2-1/3 \\ -x^3-3x^2/2+x/2+1/2 \\ x^4+4x^3+3x^2-2x-6/5
\end{array}
\right)=\left(
\begin{array}{c}
1 \\ 1/2-x \\ x^2-x+1/6 \\ -x(x-1)(2x-1)/2 \\  x^4-2x^3+x^2-1/30
\end{array}
\right).
\end{equation*}

%%%%%%%%%%%%%%%%%%%%%%%%%%%%%%%%%%%%%%%%%%%%%%%%%%%%%%%%%%%%%%%%%%%%%%%%%%%%%%%%%%%%%%%%%%%%%%%%
%==========                                                                    ===============%%
%%                     3.The lambda-Daehee Polynomials                              %%
%          ===================================================================                %%
%%%%%%%%%%%%%%%%%%%%%%%%%%%%%%%%%%%%%%%%%%%%%%%%%%%%%%%%%%%%%%%%%%%%%%%%%%%%%%%%%%%%%%%%%%%%%%%%

\section{The lambda-Daehee Numbers and Polynomials}
In this section we introduce the matrix representation for the lambda-Daehee polynomials.  Kim et al. \cite{8} introduce some results for $\lambda$-Daehee polynomial, we can derive these results in matrix representation and prove these results in simply by using matrix forms.
The $\lambda$-Daehee polynomials of the first kind can be defined by the generating function to be, Kim et al. \cite{8},
\begin{equation} \label{31}
\frac{ \lambda \log⁡(1+t)}{(1+t)^{\lambda}-1}  (1+t)^x = \sum_{n=0}^\infty D_{n,\lambda } (x) \frac{ t^n}{n!}.
\end{equation}

\noindent
When $x=0, \; D_{n, \lambda }=D_{n,\lambda} (0)$ are called the $\lambda$-Daehee numbers.
\begin{equation} \label{31a}
\frac{ \lambda  \log⁡(1+t)}{(1+t)^{\lambda}-1}  = \sum_{n=0}^\infty D_{n, \lambda } \frac{  t^n}{n!}.
\end{equation}

\noindent
It is easy to see that $D_n (x)= D_{n,1} (x)$ and $D_n = D_{n,1}.$

\noindent
Kim et al. \cite[Theorem 2]{8} obtained the following results.\\
For $ m \geq 0$, we have
\begin{equation} \label{32}
D_{m,\lambda} (x) = \sum_{l=0}^m s_1 (m,l) {\lambda}^l B_l \left(\frac{x}{\lambda} \right),
\end{equation}

\noindent
and
\begin{equation} \label{33}
{\lambda}^m B_m \left(\frac{x}{\lambda} \right)= \sum_{n=0}^m s_2 (m,n) D_{n, \lambda} (x),
\end{equation}

\noindent
we can write these results in the following matrix form
\begin{equation} \label{34}
{ \bf D}_{\lambda} (x) = {\bf S}_1 {\bf \Lambda B}\left( \frac{x}{\lambda} \right),
\end{equation}
and
\begin{equation} \label{35}
{\bf {\Lambda} } {\bf B} \left( \frac{x}{\lambda} \right)= {\bf S} _2 {\bf D}_{\lambda} (x),
\end{equation}

\noindent
where, ${\bf D}_{\lambda} (x)=\left(D_{0,λ} (x) \; D_{1,λ} (x) \; \cdots \; D_{n,λ} (x)  \right)^T$, is the $(n+1)\times 1,$ matrix for $\lambda$-Daehee polynomials of the first kind, ${\bf B} \left(\frac{x}{\lambda} \right)=\left(B_0 (\frac{x}{\lambda}) \; B_1 (\frac{x}{\lambda})\; \cdots \; B_n (\frac{x}{\lambda}) \right)^T$, is the $(n+1)\times 1,$ matrix for Bernoulli polynomials, when $x\rightarrow \frac{x}{\lambda}$ and ${\bf \Lambda}$ is the $(n+1)\times (n+1),$ diagonal matrix with elements, ${\bf \Lambda}_{ii}=\lambda^i, \;   i=0,1,\cdots,n$.

\noindent
For example, if setting $0 \leq n \leq 4$, in (\ref{34}), we have
\begin{eqnarray*}
&&
\left(
\begin{array}{ccccc}
1 & 0 & 0 & 0 & 0 \\ 0 & 1 & 0 & 0 & 0 \\ 0 & -1 & 1 & 0 & 0 \\ 0 & 2 & -3 & 1 & 0 \\ 0 & -6 & 11 & -6 & 1 \end{array}
\right)\left(
\begin{array}{ccccc}
1 & 0 & 0 & 0 & 0 \\ 0 &\lambda & 0 & 0 & 0 \\ 0 & 0 & \lambda^2 & 0 & 0 \\ 0 & 0 & 0 & \lambda^3 & 0 \\ 0 & 0 & 0 & 0 & \lambda^4 \end{array}
\right)\left(
\begin{array}{c}
1 \\ (x-\lambda/2)/\lambda \\  (\lambda^2/6-\lambda x+x^2)/\lambda^2 \\ x(\lambda-x)(\lambda-2x)/(2\lambda^3) \\ (\lambda^2 x^2-\lambda^4/30-2\lambda x^3+x^4 )/\lambda^4
\end{array}
\right)=
\\
&&
\left(
\begin{array}{c}
1 \\ x-\lambda/2 \\ \lambda^2/6-\lambda x+\lambda/2+x^2-x \\-(\lambda-2x+2)(\lambda-2x+x^2-\lambda x)/2 \\ \lambda^2 x^2-\lambda^4/30-3\lambda^2 x+11\lambda^2/6-2\lambda x^3+9\lambda x^2-11 \lambda x+3\lambda+x^4-6x^3+11x^2-6x
\end{array}
\right).
\end{eqnarray*}
\begin{remark}
In fact, we can prove Eq. (\ref{35}) by multiplying Eq. (\ref{34}) by ${\bf S}_2$ as follows.
\begin{equation*}
{\bf S}_2 {\bf D}_{\lambda} (x)={ \bf S}_2 {\bf S}_1 {\bf \Lambda B} \left(\frac{x}{\lambda} \right)={\bf \Lambda B} \left( \frac{x}{\lambda } \right).
\end{equation*}
\end{remark}
\begin{theorem}
For $m \geq 0$, we have
\begin{equation} \label{36}
D_{m, \lambda} (\lambda x)= m! \sum_{n=0}^m \sum_{ i_1 + i_2 + \cdots + i_n = m } \frac{D_n (x)}{n!}
\left(^\lambda_{i_1} \right) \left(^\lambda_{i_2} \right)\cdots \left(^\lambda_{i_n}  \right).
\end{equation}
\end{theorem}
\begin{proof}
From (\ref{1}), replacing $(1+t)$ by $(1+t)^\lambda$, we have
\begin{equation*}
\left( \frac{ \lambda \log⁡(1+t)}{(1+t)^\lambda -1} \right) (1+t)^{\lambda x} = \sum_{n=0}^\infty D_n (x)  \frac{ \left((1+t)^\lambda-1 \right)^n}{n!},
\end{equation*}

\noindent
thus from (\ref{31}), we get
\begin{eqnarray*}
\sum_{m=0}^\infty D_{m,\lambda} (\lambda x) \frac{ t^m}{m!} & = &
\sum_{n=0}^\infty \frac{D_n (x)}{n!} \left( \sum_{i=0}^\lambda \left(^\lambda_i \right) t^i -1 \right)^n
\\
& = & \sum_{n=0}^\infty \frac{D_n (x)}{n!}  \left( \sum_{i=1}^\lambda \left(^\lambda_i\right) t^i \right)^n.
\end{eqnarray*}

\noindent
Using Cauchy rule of product of series, we obtain
\begin{eqnarray*}
\sum_{m=0}^\infty D_{m,\lambda} (\lambda x) \frac{ t^m}{m!} & = &
\sum_{n=0}^\infty \frac{D_n (x)}{n!} \sum_{m=n}^\infty \sum_{i_1 + i_2 + \cdots + i_n = m }\left(^\lambda_{i_1}\right) \cdots \left(^\lambda_{i_n} \right) t^m
\\
& = & \sum_{m=0}^\infty \sum_{n=0}^m m! \sum_{i_1 + i_2 + \cdots + i_n = m} \frac{D_n (x)}{n!} \left(^\lambda_{i_1} \right)) \cdots \left(^\lambda_{i_n} \right)  \frac{t^m}{m!}.
\end{eqnarray*}

\noindent
Equating the coefficients of $t^m$ on both sides yields (\ref{36}). This completes the proof.

\end{proof}

\noindent
Setting  $x=0$, in (\ref{36}), we have the following corollary as a special case.

\begin{corollary}
For $m \geq 0$, we have
\begin{equation} \label{37}
D_{m, \lambda} = m! \sum_{n=0}^m \sum_{i_1 + i_2 + \cdots + i_n =m }\frac{D_n}{n!} \left(^\lambda_{i_1} \right) \left(^\lambda_{i_2} \right) \cdots \left(^\lambda_{i_n} \right).
\end{equation}

\end{corollary}

\noindent
Kim et al. \cite{8}, defined the $\lambda$-Daehee polynomials of the second kind as follows.
\begin{equation} \label{38}
\frac{ \lambda \log⁡(1+t)}{(1+t)^\lambda-1}  (1+t)^x = \sum_{n=0}^\infty \hat{D}_{n,\lambda}(x)\frac{ t^n}{n!}.
\end{equation}

\noindent
Note that $\hat{D}_{n,1} (x)=\hat{D}_n (x)$.

\noindent
Kim et al. \cite[Theorem 4]{8} introduced the following results.\\
For  $m \geq 0$, we have
\begin{equation} \label{39}
\hat{D}_{m,\lambda} (x) = \sum_{l=0}^m s_1 (m,l) \lambda^l B_l \left(1+\frac{x}{\lambda}\right),                                                    \end{equation}
and
\begin{equation} \label{40}
\lambda^m B_m \left(1+\frac{x}{\lambda} \right)= \sum_{n=0}^m s_2 (m,n) \hat{D}_{n,\lambda} (x),
\end{equation}

\noindent
we can write these results in the following matrix form
\begin{equation} \label{41}
{\bf \hat{D}}_{\lambda} (x)= {\bf S}_1 {\bf \Lambda B}\left(1+\frac{x}{\lambda}\right),
\end{equation}

\noindent
and
\begin{equation} \label{42}
{\bf \Lambda B} \left(1+\frac{x}{\lambda}\right)={\bf S}_2 {\bf \hat{D}}_\lambda (x),
\end{equation}

\noindent
where ${\bf \hat{D}}_\lambda (x) = \left( \hat{D}_{0,\lambda} (x) \; \hat{D}_{1, \lambda} (x) \; \cdots \; \hat{D}_{n,\lambda}(x) \right)^T$, is the $(n+1)\times 1$,  matrix for $\lambda$-Daehee polynomials of the second kind. ${\bf B} \left(1+\frac{x}{\lambda} \right)=\left( B_0 \left( 1+ \frac{x}{\lambda} \right) \; B_1 \left(1+ \frac{x}{\lambda} \right) \; \cdots \; B_n \left( 1 + \frac{x}{\lambda} \right) \right)^T$, is the $(n+1)\times 1$ matrix for Bernoulli Polynomials, when $x \rightarrow 1+\frac{x}{\lambda}$.

\noindent
For example, if setting $0 \leq n \leq 4$, in (\ref{41}), we have
\begin{footnotesize}
\begin{eqnarray*}
&&
\left(
\begin{array}{ccccc}
1 & 0 & 0 & 0 & 0 \\ 0 & 1 & 0 & 0 & 0 \\ 0 & -1 & 1 & 0 & 0 \\ 0 & 2 & -3 & 1 & 0 \\ 0 & -6 & 11 & -6 & 1 \\
\end{array}
\right)\left(
\begin{array}{ccccc}
1 & 0 & 0 & 0 & 0 \\ 0 & \lambda & 0 & 0 & 0 \\ 0 & 0 & \lambda^2 & 0 & 0 \\ 0 & 0 & 0 & \lambda^3 & 0 \\ 0 & 0 & 0 & 0 & \lambda^4
\end{array}
\right)\left(
\begin{array}{c} 1 \\ (x+\lambda/2)/\lambda \\  (x^2+\lambda x+\lambda^2/6)/\lambda^2 \\  x(\lambda+2x)(\lambda+x)/(2\lambda^3) \\  (30\lambda^2 x^2-\lambda^4+60\lambda x^3+30x^4 )/(2\lambda^3)
\end{array}
\right)=
\\
&&
\; \; \;
\left(
\begin{array}{c}
1 \\ \lambda/2+ x \\ \lambda^2/6+\lambda x-\lambda/2+x^2-x \\-(\lambda+2x-2)(\lambda+2x-x^2-\lambda x)/2 \\ \lambda^2 x^2-\lambda^4/30-3\lambda^2 x+11\lambda^2/6+2\lambda x^3-9\lambda x^2+11\lambda x-3\lambda +x^4-6x^3+11x^2-6x\\
\end{array}
\right)
\end{eqnarray*}
\end{footnotesize}
\begin{remark}
In fact, we can prove Eq. (\ref{42}) by multiplying Eq.(\ref{41}) by ${\bf S}_2$ as follows.
\begin{equation*}
{\bf S}_2 {\bf \hat{D}}_{\lambda}  (x)={\bf S}_2 {\bf S}_1 {\bf \Lambda B} \left(1+\frac{x}{\lambda} \right)= {\bf \Lambda B} \left(1+\frac{x}{\lambda} \right).
\end{equation*}
\end{remark}

%%%%%%%%%%%%%%%%%%%%%%%%%%%%%%%%%%%%%%%%%%%%%%%%%%%%%%%%%%%%%%%%%%%%%%%%%%%%%%%%%%%%%%%%%%%%%%%%
%==========                                                                    ===============%%
%%                     4.The Twisted Daehee Numbers and Polynomials                           %%
%          ===================================================================                %%
%%%%%%%%%%%%%%%%%%%%%%%%%%%%%%%%%%%%%%%%%%%%%%%%%%%%%%%%%%%%%%%%%%%%%%%%%%%%%%%%%%%%%%%%%%%%%%%%
\section{The Twisted Daehee Numbers and Polynomials}
Park et al. \cite{11} defined the nth twisted Daehee polynomials of the first kind by the generating function to be
\begin{equation} \label{43}
\left( \frac{\log⁡(1+\xi t)}{\xi t}\right) (1+\xi t)^x = \sum_{n=0}^\infty D_{n,\xi} (x)  \frac{  t^n}{n!}.
\end{equation}

\noindent
In the special case, $x=0, D_{n,\xi}  = D_{n,\xi} (0)$ are called the $n$th twisted Daehee numbers of the first kind.
\begin{equation} \label{44}
\left( \frac{\log⁡(1+\xi t)}{\xi t} \right)=\sum_{n=0}^\infty D_{n,\xi } \frac{ t^n}{n!}.
\end{equation}

\noindent
In fact, we can obtain the relation between Daehee polynomials and twisted Daehee polynomials of the first kind as follows.

\noindent
Replacing $t$ with $\xi t$ in Eq. (\ref{1}), we have
\begin{equation} \label{45}
\left( \frac{\log⁡(1+\xi t)}{ \xi t} \right) (1+\xi t)^x= \sum_{n=0}^\infty D_n (x) \frac{\xi^n t^n}{n!}.
\end{equation}

\noindent
From Equations (\ref{45}) and (\ref{43}), equating the coefficients of $t^n$ on both sides, we obtain
\begin{equation} \label{46}
D_{n, \xi} (x)= \xi^n D_n (x), \;     n \geq 0.
\end{equation}

\noindent
Setting $x=0$ in (\ref{46}) we have
\begin{equation} \label{47}
D_{n,\xi}= \xi^n D_n,   \; \;  n≥0.
\end{equation}

\noindent
We can write Eq. (\ref{46}) in the matrix form as follows.
\begin{equation} \label{48}
{\bf D}_{\xi} (x)={\bf \Xi D}(x).
\end{equation}

\noindent
where ${\bf D}_{\xi} (x)$, is the $(n+1)\times 1$, matrix for the twisted Daehee polynomials, ${\bf \Xi}$ is the $(n+1)\times (n+1)$ diagonal matrix with elements ${\bf \Xi}_{ii}=\xi^i,\; i=0,1,\cdots ,n$ , ${\bf D} (x)$, is the $(n+1)\times 1$, matrix for Daehee polynomials.

\noindent
For example, if setting $ 0 \leq n \leq 4$, in (\ref{48}), we have

\begin{footnotesize}
\begin{equation*}
\left(
\begin{array}{ccccc}
1 & 0 & 0 & 0 & 0 \\ 0 & \xi & 0 & 0 & 0 \\ 0 & 0 & \xi^2 & 0 & 0 \\ 0 & 0 & 0 & \xi^3 & 0 \\ 0 & 0 & 0 & 0 & \xi^4 \\
\end{array}
\right)\left(
\begin{array}{c}
1 \\ x-1/2 \\ x^2-2x+2/3 \\ x^3-9x^2/2+11x/2-3/2 \\ x^4-8x^3+21x^2-20x+24/5 \\
\end{array}
\right)=\left(
\begin{array}{c}
1 \\ \xi(x-1/2) \\ \xi^2 (x^2-2x+2/3) \\ \xi^3 (2x-3)(x^2-3x+1)/2 \\ \xi^4 (x^4-8x^3+21x^2-20x+24/5)
\end{array}
\right)
\end{equation*}
\end{footnotesize}

\noindent
From Eq. (\ref{47}) in Equations (\ref{5}) and (\ref{6}), we have
\begin{equation} \label{49}
D_{n,\xi}=(-1)^n \frac{ n!}{n+1}  \xi^n, \; n \geq 0.
\end{equation}

\noindent
and
\begin{equation} \label{50}
(n+1) D_{n,\xi})+n^2 \xi D_{n-1, \xi })=0, \; \;    n \geq 1.
\end{equation}

\noindent
Substituting  from Eq. (\ref{47}) in Eq. (\ref{9}), we have
\begin{equation} \label{51}
D_{n,\xi} (1+x) = D_{n,\xi } (x)+ n \xi^{-1} D_{n-1, \xi} (x).
\end{equation}

\noindent
Park et al. \cite[Eq. 17]{11} proved the following relation
\begin{equation} \label{52}
D_{n,\xi}=\xi^n \sum_{l=0}^n s_1 (n,l) B_l,
\end{equation}

\noindent
which can be written in the following matrix form
\begin{equation} \label{53}
{\bf D}_{\xi} ={\bf \Xi S}_1 {\bf B},
\end{equation}

\noindent
where ${\bf D}_{\xi}=\left(D_{0,\xi} \; D_{1,\xi } \; \cdots \; D_{n, \xi} \right)^T$ is the $(n+1)\times 1$, matrix for the twisted Daehee numbers,  ${\bf S}_1$  and ${\bf B}$ are matrices for Striling numbers of the first kind and Bernoulli numbers, respectively. ${\bf \Xi}$ is the $(n+1)\times (n+1)$ diagonal matrix.

\noindent
For example, if setting  $0 \leq n \leq 4$, in (\ref{53}), we have
\begin{equation*}
\left(
\begin{array}{ccccc}
1 & 0 & 0 & 0 & 0 \\ 0 & \xi & 0 & 0 & 0 \\ 0 & 0 & \xi^2 & 0 & 0 \\ 0 & 0 & 0 & \xi^3 & 0 \\ 0 & 0 & 0 & 0 & \xi^4\\
\end{array}
\right)\left(
\begin{array}{ccccc}
1 & 0 & 0 & 0 & 0 \\ 0 & 1 & 0 & 0 & 0 \\ 0 & -1 & 1 & 0 & 0 \\ 0 & 2 & -3 & 1 & 0 \\ 0 & -6 & 11 & -6 & 1\\
\end{array}
\right)\left(
\begin{array}{c}
1 \\ -1/2 \\ 1/6 \\ 0 \\ -1/30 \\
\end{array}
\right)=\left(
\begin{array}{c}
1 \\ -\xi/2 \\ 2\xi^2/3 \\ -3\xi^3/2 \\ 24\xi^4/5
\end{array}
\right).
\end{equation*}

\noindent
Park et al. \cite[Corollary 3]{11} proved the following result for twisted Daehee polynomials of the first kind.

\noindent
For $n \geq 0$, we have
\begin{equation} \label{54}
D_{n,\xi} (x)=\xi^n \sum_{l=0}^n s_1 (n,l) B_l (x),
\end{equation}

\noindent
we can write Eq. (\ref{54}) in the matrix form as follows.
\begin{equation} \label{55}
{\bf D}_\xi (x)={\bf \Xi S}_1 {\bf B}(x),
\end{equation}

\noindent
where ${\bf D}_\xi (x)=\left(D_{0,\xi} (x) \; D_{1,\xi } (x) \; \cdots \; D_{n,\xi} (x) \right)^T$ is the $(n+1)\times 1$, matrix for the twisted Daehee polynomials,  ${\bf B} (x)$ is the $(n+1)\times 1$, matrix for Bernoulli polynomials.

\noindent
For example, if setting $0 \leq n \leq 4$, in (\ref{55}), we have
\begin{eqnarray*}
&&
\left(
\begin{array}{ccccc}
1 & 0 & 0 & 0 & 0 \\ 0 & \xi & 0 & 0 & 0 \\ 0 & 0 & \xi^2 & 0 & 0 \\ 0 & 0 & 0 & \xi^3 & 0 \\ 0 & 0 & 0 & 0 & \xi^4\\
\end{array}
\right)\left(
\begin{array}{ccccc}
1 & 0 & 0 & 0 & 0 \\ 0 & 1 & 0 & 0 & 0 \\ 0 & -1 & 1 & 0 & 0 \\ 0 & 2 & -3 & 1 & 0 \\ 0 & -6 & 11 & -6 & 1\\
\end{array}
\right) \left(
\begin{array}{c}
1 \\ x-1/2 \\ x^2-x+1/6 \\ x^3-3x^2/2+x/2 \\ x^4-2x^3+x^2-1/30\\
\end{array}
\right)=
\\
&&
\; \; \; \; \; \; \; \; \; \; \; \; \; \; \; \; \; \; \; \; \;
\left(
\begin{array}{c}
1 \\ \xi (x-1/2) \\ \xi^2 (x^2-2x+2/3) \\ \xi^3 (2x-3)(x^2-3x+1)/2 \\ \xi^4 (x^4-8x^3+21x^2-20x+24/5) \\
\end{array}
\right).
\end{eqnarray*}

\noindent
The twisted Bernoulli polynomials are defined as,  Kim \cite{6},
\begin{equation} \label{56}
\left( \frac{t}{\xi e^t-1} \right)e^{xt}= \sum_{n=0}^\infty B_{n,\xi } (x) \frac{ t^n}{n!},  \;   \xi \in T_p,
\end{equation}

\noindent
and the twisted Bernoulli numbers $B_{n,\xi}$ are defined as $B_{n,\xi}  = B_{n,\xi} (0)$.
\begin{equation} \label{56a}
\frac{t}{ \xi e^t-1}= \sum_{n=0}^\infty B_{n,\xi}   \frac{t^n}{n!}, \;    \xi \in T_p.
\end{equation}

\noindent
We introduce the relation between the twisted Daehee polynomials of the first kind and the Bernoulli polynomials in the following Theorem.

\begin{theorem} \label{T4-1}
For $m \geq 0$, we have
\begin{equation} \label{57}
B_m (x) = \sum_{n=0}^m s_2 (m,n) \xi^{-n} D_{n, \xi} (x).
\end{equation}
\end{theorem}
\begin{proof}
Replace $t$ by $(e^t-1)/\xi$, in (\ref{43}), we have
\begin{eqnarray} \label{58}
\left( \frac{t}{e^t-1} \right) e^{tx} & = &
\sum_{n=0}^\infty D_{n,\xi} (x)  \frac{(e^t-1)^n}{n!}  \xi^{-n}
\nonumber\\
& = & \sum_{n=0}^\infty D_{n,\xi} (x) \frac{\xi^{-n}}{n!}  n! \sum_{m=n}^\infty s_2 (m,n)  \frac{t^m}{m!}
\nonumber\\
& = & \sum_{m=0}^\infty \sum_{n=0}^m \xi^{-n} D_{n,\xi} (x) s_2 (m,n) \frac{t^m}{m!}.
\end{eqnarray}

\noindent
From (\ref{4}), when $\alpha=1$, we get
\begin{equation} \label{59}
\left( \frac{t}{e^t-1} \right) e^{xt} = \sum_{m=0}^\infty B_m (x) \frac{t^m}{m!},
\end{equation}

\noindent
From  (\ref{58}) and (\ref{59}), we have
\begin{equation*}
B_m (x) = \sum_{n=0}^m \xi^{-n} D_{n,\xi} (x) s_2 (m,n).
\end{equation*}

\noindent
This completes proof (\ref{57}).
\end{proof}

\begin{remark}
In fact, it seems that there is something not correct in Theorem 4, Park et al. \cite{11}.
\end{remark}

\noindent
Moreover, we can represent Equation (\ref{57}), in the following matrix form.
\begin{equation} \label{60}
{\bf B} (x)={\bf S}_2  {\bf \Xi}^{-1} {\bf D}_\xi (x).
\end{equation}
\begin{remark}
We can prove (\ref{54}), Park et al. \cite[Corollary 3]{11}, easily using the matrix form as follows.
Multiplying the both sides of (\ref{60}) by ${\bf S}_1$, we get
\begin{equation*}
{\bf S}_1 {\bf B}(x)={\bf S}_1 {\bf S}_2  {\bf \Xi}^{-1} {\bf D}_\xi (x)= { \bf I} {\bf \Xi}^{-1} {\bf D}_\xi (x)={\bf \Xi}^{-1} {\bf D}_\xi (x)
\end{equation*}
and
\begin{equation*}
{\bf \Xi S}_1 {\bf B}(x)= {\bf \Xi \Xi}^{-1} {\bf D}_\xi (x)={\bf D}_\xi (x).
\end{equation*}

\noindent
This completes the proof of (\ref{54}).
\end{remark}

\noindent
For example, if setting $0 \leq n \leq 4 $, in (\ref{60}), we have
\begin{footnotesize}
\begin{eqnarray*}
&&
\left(
\begin{array}{ccccc}
 1 & 0 & 0 & 0 & 0 \\ 0 & 1 & 0 & 0 & 0 \\ 0 & 1 & 1 & 0 & 0 \\ 0 & 1 & 3 & 1 & 0 \\ 0 & 1 & 7 & 6 & 1 \\
\end{array}
\right)\left(
\begin{array}{ccccc}
 1 & 0 & 0 & 0 & 0 \\ 0 & \xi^{-1} & 0 & 0 & 0 \\ 0 & 0 & \xi^{-2} & 0 & 0 \\ 0 & 0 & 0 & \xi^{-3} & 0 \\ 0 & 0 & 0 & 0 & \xi^{-4}\\
\end{array}
\right)
\left(
\begin{array}{c}
1 \\ \xi(x-1/2) \\ \xi^2 (x^2-2x+2/3) \\ \xi^3  (2x-3)(x^2-3x+1)/2 \\ \xi^4 (x^4-8x^3+21x^2-20x+24/5) \\
\end{array}
\right)
=
\\
&&
\; \; \; \; \; \; \; \; \; \; \; \; \; \; \; \; \; \; \; \; \; \; \; \; \; \; \; \; \; \;
\left(
\begin{array}{c}
1\\ x-1/2 \\ x^2-x+1/6 \\ x(x-1)(2x-1)/2 \\ x^4-2x^3+x^2-1/30))\\
\end{array}
\right).
\end{eqnarray*}
\end{footnotesize}

\noindent
Setting $x=0$ in Theorem \ref{T4-1}, we have the following Corollary.
\begin{corollary}
For $m \geq 0$, we have
\begin{equation}
B_m = \sum_{n=0}^m s_2 (m,n) \xi^{-n} D_{n, \xi}.
\end{equation}
\end{corollary}

\noindent
We can investigate a new relation between the twisted Bernoulli polynomials and Bernoulli polynomials as follows.
\begin{theorem} \label{T4-3}
For $n\geq 0$, we have
\begin{equation} \label{61}
\xi^x B_{n,\xi} (x)+\xi^x  \ln⁡(\xi) B_{n+1,\xi} (x)  \left( \frac{1}{n+1}\right)=\sum_{j=n}^\infty B_j (x) \frac{(\ln⁡(\xi) )^{j-n}}{(j-n)!}.
\end{equation}
\end{theorem}
\begin{proof}
From (\ref{4}), when $\alpha=1$, we get
\begin{equation*}
\left( \frac{t}{e^t-1} \right) e^{xt}=\sum_{m=0}^\infty B_m (x)\frac{t^m}{m!},
\end{equation*}

\noindent
and replace $t$ by $(t+\ln \xi )$, we have
\begin{eqnarray*}
\left(\frac{t+\ln⁡ξ}{ \xi e^t-1} \right) e^{x(t+\ln \xi ξ)}  & = & \sum_{m=0}^\infty B_m (x) \frac{(t+\ln⁡ \xi)^m}{m!},
\\
\left(  \frac{t}{\xi e^t-1}+\frac{\ln \xi }{\xi e^t-1} \right) e^{xt} \xi^x & = & \sum_{m=0}^\infty \frac{ B_m (x)}{m!} \sum_{i=0}^m \left(^m_i\right) (\ln \xi )^{m-i} t^i ,
\\
\xi^x \left( \frac{t}{\xi e^t-1} \right) e^{xt}+\frac{\xi^x  \ln\xi}{t} \left( \frac{t}{ \xi e^t-1} \right) e^{xt} & = & \sum_{i=0}^\infty \sum_{m=i}^\infty \frac{B_m (x))}{m!} \left(^m_i\right) (\ln \xi )^{m-i} t^i ,
\\
\xi^x \sum_{n=0}^\infty B_{n,\xi} (x)  \frac{t^n}{n!} + \frac{\xi^x  \ln\xi }{t} \sum_{n=0}^\infty B_{n,\xi} (x)  \frac{t^n}{n!} & = & \sum_{i=0}^\infty \sum_{m=i}^\infty \frac{B_m (x)}{(m-i)!} (\ln \xi )^{m-i}  \frac{t^i}{i!}.
\end{eqnarray*}

\noindent
By equating the coefficients of $t^n$ on both sides gives
\begin{equation*}
\xi^x B_{n,\xi} (x)+\left( \frac{\xi^x  \ln⁡(\xi) }{n+1 } \right) B_{n+1,\xi} (x) = \sum_{m=n}^\infty \frac{(\ln \xi )^{m-n}}{(m-n)!}  B_m (x), \; n\geq 0.
\end{equation*}

\noindent
This completes the proof.
\end{proof}

\noindent
Setting $x=0$ in Theorem \ref{T4-3}, we have the following Corollary.
\begin{corollary}
For  $n \geq 0$, we have
\begin{equation} \label{62}
B_{n,\xi}+\left( \frac{\ln⁡(\xi)}{n+1 } \right) B_{n+1,\xi}= \sum_{m=n}^\infty \frac{(\ln \xi  )^{m-n}}{(m-n)!}  B_m, \;   n\geq 0.
\end{equation}
\end{corollary}

\noindent
Equation (\ref{62}) gives a new connection between twisted Bernoulli numbers and Bernoulli numbers.\\

\noindent
Park et al. \cite{11} introduced the $n$th twisted Daehee polynomials of the second kind as follows:
\begin{equation} \label{63}
\left( \frac{(1+\xi t)  \log⁡(1+\xi t)}{ \xi t}  \right) \frac{1}{(1+\xi t)^x} = \sum_{n=0}^\infty \hat{D}_{n,\xi} (x) \frac{t^n}{n!}.
\end{equation}

\noindent
In special, if $x=0, \; \hat{D}_{n,\xi }= \hat{D}_{n,ξ} (0)$, we have the twisted Daehee numbers of second kind.
\begin{equation} \label{63a}
\frac{(1+\xi t)  \log⁡(1+\xi t)}{\xi t} = \sum_{n=0}^\infty \hat{D}_{n,\xi} \frac{t^n}{n!}.
\end{equation}

\noindent
Park et al. \cite[Theorem 5]{11}, proved the following theorem.\\
For $n \geq 0$, we have
\begin{equation} \label{64}
\hat{D}_{n,\xi }=\xi^n \sum_{l=0}^n s_1 (n,l)  (-1)^l B_l.
\end{equation}

\noindent
We can write this theorem in the following matrix form.
\begin{equation} \label{65}
{\bf \hat{D}}_\xi= {\bf \Xi S}_1 {\bf I}_1 {\bf B},
\end{equation}

\noindent
where ${\bf \hat{D}}_\xi=\Big( \hat{D}_{0,\xi} \; \hat{D}_{1,\xi } \; \cdots  \; \hat{D}_{n,\xi } \Big)$ is the $(n+1)\times 1$, matrix for twisted Daehee numbers of the second kind and ${\bf I}_1$ is the $(n+1)\times (n+1)$ diagonal matrix with elements $({\bf I}_1)_{ii}=(-1)^i, i=0,1,\cdots,n$.

\noindent
For example, if setting $0\leq n\leq 4$ , in (\ref{65}), we have
\begin{footnotesize}
\begin{equation*}
\left(
\begin{array}{ccccc}
1 & 0 & 0 & 0 & 0 \\ 0 & \xi & 0 & 0 & 0 \\ 0 & 0 & \xi^2 & 0 & 0 \\ 0 & 0 & 0 & \xi^3 & 0 \\ 0 & 0 & 0 & 0 & \xi^4\\
\end{array}
\right)\left(
\begin{array}{ccccc}
 1 & 0 & 0 & 0 & 0 \\ 0 & 1 & 0 & 0 & 0 \\ 0 & -1 & 1 & 0 & 0 \\ 0 & 2 & -3 & 1 & 0 \\ 0 & -6 & 11 & -6 & 1\\
\end{array}
\right)\left(
\begin{array}{ccccc}
1 & 0 & 0 & 0 & 0 \\ 0 & -1 & 0 & 0 & 0 \\ 0 & 0 & 1 & 0 & 0 \\ 0 & 0 & 0 & -1 & 0 \\ 0 & 0 & 0 & 0 & 1 \\
\end{array}
\right)
\left(
\begin{array}{c}
1 \\ -1/2 \\ 1/6 \\ 0 \\ -1/30 \\
\end{array}
\right)=\left(
\begin{array}{c}
1 \\ \xi/2 \\ -\xi^2/3 \\ \xi^3/2 \\-6\xi^4/5\\
\end{array}
\right).
\end{equation*}
\end{footnotesize}

\noindent
Park et al. \cite[Theorem 6]{11} proved the following theorem.\\
For $n \geq 0$, we have
\begin{equation} \label{66}
\hat{D}_{n,\xi} (x)=\xi^n \sum_{l=0}^n s_1 (n,l) (-1)^l B_l (x).
\end{equation}

\noindent
We can write this theorem in the following matrix form
\begin{equation} \label{67}
{\bf \hat{D}}_\xi (x)={\bf \Xi S}_1 {\bf I}_1 {\bf B}(x),
\end{equation}

\noindent
where ${\bf \hat{D}}_\xi (x)=\left( \hat{D}_{0,\xi} (x) \; \hat{D}_{1,\xi} (x) \; \cdots  \; \hat{D}_{n,\xi } (x) \right)$ is the$(n+1)\times 1$, matrix for twisted Daehee polynomials of the second kind.

\noindent
For example, if setting  $0 \leq n \leq 4$, in (\ref{67}), we have
\begin{scriptsize}
\begin{eqnarray*}
&&
\left(
\begin{array}{ccccc}
1 & 0 & 0 & 0 & 0 \\ 0 & \xi & 0 & 0 & 0 \\ 0 & 0 & \xi^2 & 0 & 0 \\ 0 & 0 & 0 & \xi^3 & 0 \\ 0 & 0 & 0 & 0 & \xi^4\\
\end{array}
\right)\left(
\begin{array}{ccccc} 1 & 0 & 0 & 0 & 0 \\ 0 & 1 & 0 & 0 & 0 \\ 0 & -1 & 1 & 0 & 0 \\ 0 & 2 & -3 & 1 & 0 \\ 0 & -6 & 11 & -6 & 1 \\
\end{array}
\right)\left(
\begin{array}{ccccc}
1 & 0 & 0 & 0 & 0 \\ 0 & -1 & 0 & 0 & 0 \\ 0 & 0 & 1 & 0 & 0 \\ 0 & 0 & 0 & -1 & 0 \\ 0 & 0 & 0 & 0 & 1\\
\end{array}
\right)\left(
\begin{array}{c}
1 \\ x-1/2 \\ x^2-x+1/6 \\ x^3-(3x^2)/2+x/2 \\ x^4-2x^3+x^2-1/30\\
\end{array}
\right)=
\\
&&
\; \; \; \; \; \; \; \; \; \; \; \; \; \; \; \; \; \; \; \; \; \;
\left(
\begin{array}{c}
1 \\ -\xi (x-1/2) \\ \xi^2 (x^2-1/3) \\-(\xi^3/2)(2x+1)(x^2+x-1) \\ \xi ^4 (x^4+4x^3+3x^2-2x-6/5) \\
\end{array}
\right).
\end{eqnarray*}
\end{scriptsize}

\noindent
We can obtain the relation between Daehee polynomials and twisted Daehee polynomials of the second kind as follows.
\begin{proposition}
For $n\geq 0$, we have
\begin{equation} \label{68}
\hat{D}_{n,\xi} (x)=\xi^n \hat{D}_n (x), \;     n \geq0.
\end{equation}
\end{proposition}
\begin{proof}
Replacing $t$ with $\xi t$ in (\ref{18}), we have
\begin{equation} \label{69}
\left( \frac{(1+ \xi t)  \log⁡(1+\xi t)}{\xi t} \right)  \frac{1}{(1+\xi t)^x} = \sum_{n=0}^\infty \hat{D}_n (x)  \frac{ \xi^n t^n}{n!}.
\end{equation}

\noindent
From Equations (\ref{69}) and (\ref{63}), equating the coefficients of $t^n$ on both sides, we obtain  (\ref{68}).
\end{proof}

\noindent
We can write Eq. (\ref{68}) in the matrix form as follows.
\begin{equation} \label{70}
{\bf \hat{D}}_\xi (x) = {\bf \Xi \hat{D}}(x).
\end{equation}

\noindent
where  ${\bf \hat{D}}_\xi (x)$, is the $(n+1)\times 1$, matrix for the twisted Daehee polynomials of the second kind, ${\bf \Xi} $ is the $(n+1)\times (n+1)$ diagonal matrix, ${\bf \hat{D}}(x)$, is the $(n+1)\times 1$,  matrix for Daehee polynomials of the second kind.

\noindent
For example, if setting $0 \leq n \leq 4$, in (\ref{70}), we have
\begin{equation*}
\left(
\begin{array}{ccccc}
1 & 0 & 0 & 0 & 0 \\ 0 & \xi & 0 & 0 & 0 \\ 0 & 0 & \xi^2 & 0 & 0 \\ 0 & 0 & 0 & \xi^3 & 0 \\ 0 & 0 & 0 & 0 & \xi^4\\
\end{array} \right)\left(
\begin{array}{c}
1 \\-x+1/2 \\ x^2-1/3 \\ -x^3-3x^2/2+x/2+1/2 \\ x^4+4x^3+3x^2-2x-6/5\\
\end{array}
\right)=\left(
\begin{array}{c}
1 \\ -\xi (x-1/2) \\ \xi^2 (x^2-1/3) \\ -(\xi^3/2)(2x+1)(x^2+x-1) \\ \xi^4 (x^4+4x^3+3x^2-2x-6/5)
\\
\end{array}
\right).
\end{equation*}

%%%%%%%%%%%%%%%%%%%%%%%%%%%%%%%%%%%%%%%%%%%%%%%%%%%%%%%%%%%%%%%%%%%%%%%%%%%%%%%%%%%%%%%%%%%%%%%%
%==========                                                                    ===============%%
%%                                     Refrences                                              %%
%          ===================================================================                %%
%%%%%%%%%%%%%%%%%%%%%%%%%%%%%%%%%%%%%%%%%%%%%%%%%%%%%%%%%%%%%%%%%%%%%%%%%%%%%%%%%%%%%%%%%%%%%%%%

\bibliographystyle{plain}

\end{document}